\documentclass{amsart}

\usepackage{bbm}
\usepackage{amsmath}
\usepackage{amssymb}
\usepackage[all]{xy}
\usepackage{enumerate}

\newtheorem{theorem}{Theorem}[section] 
\newtheorem{lemma}[theorem]{Lemma}     
\newtheorem{corollary}[theorem]{Corollary}
\newtheorem{proposition}[theorem]{Proposition}

\theoremstyle{definition}
\newtheorem{definition}{Definition}
\newtheorem{remark}{Remark}
\newtheorem{example}{Example}

\providecommand{\setB}{\mathbb{B}}

\providecommand{\setN}{\mathbb{N}}
\providecommand{\setQ}{\mathbb{Q}}
\providecommand{\setR}{\mathbb{R}}
\providecommand{\setS}{\mathbb{S}}
\providecommand{\setU}{\mathbb{U}}
\providecommand{\setZ}{\mathbb{Z}}

\providecommand{\tensor}{\otimes}
\providecommand{\eps}{\varepsilon}

\providecommand{\inoo}[1]{\left({#1 }\right)}
\providecommand{\incc}[1]{\left[{#1 }\right]}
\providecommand{\calC}{\mathcal{C}}

\providecommand{\calN}{\mathcal{N}}
\providecommand{\calU}{\mathcal{U}}
\providecommand{\calV}{\mathcal{V}}
\providecommand{\calW}{\mathcal{W}}
\providecommand{\conj}[1]{\overline{#1 }}
\providecommand{\quot}[2]{{#1}-slash/{#2}}

\providecommand{\frakX}{\mathfrak{X}}
\providecommand{\abs}[1]{\left\vert #1 \right\vert }
\providecommand{\diff}[1]{\frac{\mathrm{d}}{\mathrm{d}#1}}
\providecommand{\lra}{\longrightarrow }

\providecommand{\ssmall}[1]{\mbox{\begin{scriptsize}$#1$\end{scriptsize}}}

\DeclareMathOperator{\id}{id}
\DeclareMathOperator{\fix}{Fix}
\title[Equivariant Genericity]{Genericity in equivariant dynamical systems and equivariant Fuller index theory}
\author{Philipp Wruck}

\begin{document}
\begin{abstract}
We define a notion of equivariant non-degeneracy of $G$-maps to introduce the class of equivariantly non-degenerate flows on smooth compact manifolds with compact Lie group action. 
We prove genericity of this class and use this result to construct an equivariant version of the Fuller index, which detects group orbits of periodic orbits of the flow, distinguished by 
their isotropy.
\end{abstract}

\maketitle

\section{Introduction}
In 1967, B. Fuller in \cite{fuller} defined a rational homotopy invariant for flows that has properties similar to those of the classical fixed point index, as is discussed in e.g. 
\cite{nussbaum}. Locally, Fuller's index is the fixed point index of a Poincar\'{e} map for a periodic orbit, multiplied with the inverse $\frac 1k$ of its multiplicity. Chow and 
Mallet--Paret gave a construction of Fuller's index using bifurcation theory in \cite{chowmallet}. In particular, they clarified the dynamical reason why the correcting factor $\frac 1k$ 
has to be introduced, which is the occurrence of period multiplying bifurcations. 

In the last decades, there have been considerable advances in the theory of indices for equivariant maps. Ize and Vignoli \cite{izevignoli} defined an equivariant mapping degree, compare 
also \cite{balanov}. Equivariant fixed point indices have been constructed by L\"{u}ck and Rosenberg in \cite{lueck} for discrete groups and independently by Dzedzej in \cite{dzedzej} 
and Chorny \cite{ynrohc} for arbitrary compact Lie groups. The basic feature of such invariants is the property that, in addition to the classical properties of an index, they can detect 
the symmetries of fixed points (or zeros, in the case of the equivariant degree). 

In this paper we are going to construct a Fuller-like index for equivariant flows. This will be an element in the group $\bigoplus_{(H)}\setQ$, where $(H)$ runs through the orbit types of 
the underlying $G$-manifold $M$. It will be invariant under equivariant homotopies of vector fields whose periodic orbits have a prescribed period in a compact interval, bounded away from 
zero. Furthermore, non-vanishing of its $(H)$-component will imply existence of a periodic orbit of the field of isotropy type at least $(H)$.  

The construction of the index requires methods of generic equivariant bifurcation theory, which we will develop first. Though most of these results are well-known to experts, it is hard to 
find a concrete reference. Therefore, we develop the generic equivariant bifurcation picture from scratch. The point of view taken here is that genericity should be viewed as a sort of 
equivariant transversality of a graph map to the diagonal. Crucial results like the genericity of finiteness of fixed points and finiteness of bifurcation parameters are obtained easily 
from the theory of equivariant transversality. This abstract point of view allows to generalize the non-equivariant construction of Chow and Mallet--Paret almost verbatim, after carefully 
checking that every step is admissible. We also emphasize that the construction is not possible when using the notion of $G$-hyperbolicity. The generality of equivariant non-degeneracy is 
crucial to obtain the generic equivariant bifurcation picture of homotopies needed for this approach. 

The paper is organized as follows. After some preliminaries, we will define equivariant Poincar\'{e} systems in Section 3 to be able to investigate vector fields by means of associated 
equivariant Poincar\'{e} maps. In Section 4, we review the theory of transversality to stratified sets, which is the basis for the theory of equivariant transversality, and we quote the 
corresponding genericity result, the Thom-Mather Theorem \ref{thm:thommatherstrat}.

Section 5 is central to the paper. We define the notion of equivariant non-degeneracy and prove the corresponding equivariant Thom-Mather Theorem \ref{thm:equithommathernondeg}. We also 
establish via Theorem \ref{cor:equicodimdisc} that $G$-homotopies generically avoid certain high-codimensional subsets. This is crucial for the applications to fixed point theory.

In Section 6, we use a different, directly geometric approach to establish genericity theorems for vector fields and their homotopies. The basic method is to reduce the problems to
investigations of equivariant Poincar\'{e} maps, and then to apply the results of Section 5 to these maps. Though the theory of equivariant non-degeneracy helps on the level of Poincar\'{e}
maps, it can not be used directly for vector fields, since we are dealing with flows instead of general $G$-maps. Theorem \ref{thm:equigenvect} together with Theorem 
\ref{thm:equithommathernondeg} establishes a complete bifurcation picture for equivariant maps and equivariant vector fields. Generically, finitely many branches of critical elements 
undergo finitely many bifurcations.

In Section 7, we will give a definition of an equivariant Fuller index based on the techniques of Chow and Mallet--Paret. Then we will use the generic bifurcation picture to prove that the 
equivariant Fuller index is indeed $G$-homotopy invariant.


\section{Preliminaries} 
We begin by establishing some notions and notations from the theory of equivariant topology and dynamical systems.

In this paper, $G$ will always denote a compact Lie group. An action of $G$ on a topological space $X$ is a map $G\times X\to X$, denoted by $(g, x)\mapsto g.x$, such that $(gh).x=g.(h.x)$ 
and $e.x=x$ for all $g, h\in G$, $x\in X$. We sometimes also deal with maps $X\times G\to X$ with the obvious modifications and then speak of left and right $G$-actions, respectively. Maps 
between topological spaces will always be continuous, and maps between smooth manifolds will always be smooth, unless stated otherwise explicitly .

The following spaces will turn up frequently. The orbit of a point $x\in X$ is the set 
\[
Gx=\{g.x\;|\;g\in G\}.
\]
The isotropy subgroup $G_x$ of a point $x\in X$ is the group
\[
 G_x=\{g\in G\;|\;g.x=x\}.
\]
There is the well-known homeomorphism $\quot G{G_x}\cong Gx$, which is a diffeomorphism if $X$ is a manifold. If $H$ is any closed subgroup of $G$, we let $(H)$ denote its equivalence class 
under conjugacy, i.e. $H\cong K\Longleftrightarrow\;\exists\;g\in G:\,gHg^{-1}=K$. The isotropy type of a point $x\in X$ is the conjugacy class $(G_x)$ of its isotropy subgroup. For 
$x\in X$, $G_{gx}=gG_xg^{-1}$, and hence, $(G_x)=(G_{gx})$. There is a partial order on the set of isotropy types, or more general on the set of conjugacy classes of closed subgroups. It 
is defined by
\[
 (H)\leq(K)\Longleftrightarrow\,\exists\;g\in G\,:gHg^{-1}\subseteq K.
\]
With this notion, we can define various fixed spaces in $X$, namely
\[
 X^H=\{x\in X\;|\;h.x=x\;\forall\,h\in H\},
\]
\[
 X_{(H)}=\{x\in X\;|\;(G_x)=(H)\},
\]
\[
 X_{\geq(H)}=\{x\in X\;|\;(G_x)\geq(H)\}.
\]
Intersections of these spaces will be labelled in the obvious way, e.g. $X^H_{(H)}$. If $X$ is a manifold, the sets $X^H$ and $X_{(H)}$ are submanifolds of $X$. In general, $X_{(H)}$ is 
$G$-invariant and $X^H$ admits an action of the normalizer $N(H)$ of $H$ in $G$. Since $H$ fixes $X^H$, there is an induced action of the so called Weyl group $W(H)=\quot{N(H)}H$. This is 
not to be confused with the Weyl group of a Lie group $G$, which is, in our words, the Weyl group of a maximal torus $T\subseteq G$.

An important construction we will need is the twisted product of a right $H$-space $X$ with a left $H$-space $Y$. The product $X\times Y$ carries the (left) $H$-action $h.(x, y)=(x.h^{-1}, 
h.y)$, and the quotient space is denoted by $X\times_H Y$. If $X$ also carries a left $G$-action that is compatible with the $H$-action, namely $(g.x).h=g.(x.h)$ for all 
$x\in X$, $g\in G$, $h\in H$, then $X\times_H Y$ carries a left $G$-action by $g.[x,y]=[g.x, y]$. Here, $[x,y]$ denotes the equivalence of $(x, y)$ in the group quotient. In most cases, 
we will have $X=G$, $H$ a subgroup of $G$, and the actions on $G$ will be left- and right translation. 

Most prominently, twisted products turn up as tubular neighbourhoods of $G$-orbits. Every $G$-orbit $Gx$ in a $G$-manifold $M$ has a neighbourhood $G$-diffeomorphic to a space 
$G\times_HV$, where $V$ is an $H$-representation. If $M$ is Riemannian, $V$ can be taken to be the normal space of $Gx$ at $x$, that is, the orthogonal complement of $T_xGx$ in $T_xM$. 
The image of $V$ under such a diffeomorphism is called a slice at $x$, in the Riemannian case it is called a normal slice. The neighbourhood itself is called a tubular neighbourhood of $Gx$.
From e.g. \cite{bredon}, it is well known that any compact $G$-manifold with $G$ a compact Lie group admits the structure of a Riemannian manifold such that $G$ acts via isometries. We 
will therefore assume henceforth that the Riemannian structure is implicitly included in the notion of a $G$-manifold. In particular, we will be able to speak of orthogonality of 
subspaces of tangential spaces. For more details on the general theory of $G$-manifolds, see \cite{bredon}. We refer to this book as well when it comes to structural results of 
$G$-spaces and $G$-manifolds.

In the following, we assume that $X=M$ is a manifold and all maps under consideration are smooth.

From the dynamical viewpoint, we will be interested in fixed points of maps and periodic orbits of flows. A flow is nothing else than a (right) action of $\setR$ on $M$, so an equivariant 
flow $\varphi:M\times\setR\to M$ gives rise to a $(G\times\setR)$-action on $M$ by $(g, t).x=\varphi(g.x, t)$. Fixed points and periodic orbits are called the critical elements of either 
the map $f$ or the flow $\varphi$, respectively. It is obvious that the sets of fixed points and periodic orbits of an equivariant map or an equivariant vector field are $G$-invariant. 

In addition to the above, a periodic orbit, for us, is always to be considered together with a fixed period $T>0$ so that $\varphi(x, T)=x$. Furthermore, we assume that all periodic orbits 
have a minimal $p>0$ satisfying this condition, i.e. the periodic orbit is not a fixed point of the flow. If we have to investigate the underlying subspace of $M$, given by 
$\gamma=(G\times\setR).x$, we speak of the geometric periodic orbit. 

All our genericity investigations will use the Whitney topology on mapping spaces of smooth maps. For smooth manifolds $M, N$, we will denote such spaces by $\calC^\infty(M, N)$. The 
subspace of equivariant maps with respect to a smooth $G$-action is denoted by $\calC^\infty_G(M, N)$. In most cases, when the involved spaces are compact, the Whitney topology agrees 
with the $\calC^\infty$-topology of uniform convergence of the map itself and all its derivatives. 

A simple notion of non-degeneracy for equivariant critical elements is provided by the notion of $G$-hyperbolicity. Recall from \cite{field} or \cite{hirsch} that a fixed $G$-orbit of an 
equivariant self map is $G$-hyperbolic, if the tangential bundle over the orbit splits into a central bundle, corresponding to the directions of the group action, an expanding and a 
contracting bundle. Once we have equivariant Poincar\'{e} systems at hand, this concept can be carried over to define $G$-hyperbolic periodic orbits.

It is shown in \cite{field} that $G$-hyperbolicity is generic for $G$-maps and $G$-vector fields. $G$-hyperbolic critical elements are isolated and depend continuously on the map or flow. 
Furthermore, the orbit type of $G$-hyperbolic critical elements is locally constant. The notion of $G$-hyperbolicity will be too rigid for our needs, for example it is not a simple task to
define $G$-hyperbolicity of homotopies. We will however sometimes use the aforementioned genericity results.


\section{Equivariant Poincar\'{e} Systems} 
Our method of investigating periodic orbits relies on the notion of equivariant Poincar\'{e} systems, which shall be developed in the following. In contrast to the common 
definition of equivariant Poincar\'{e} systems, our systems will in general not be centered at a $G$-orbit that is contained in a periodic orbit. The reason for this is that later 
on, we want to be able to push such a system along a branch of fixed points of the associated Poincar\'{e} maps. If these maps undergo a saddle-node bifurcation, we will have fixed points 
on one side of the critical parameter, but no fixed points on the other side. We still want to be able to grasp this situation by means of Poincar\'{e} systems. Therefore, our systems do 
not start with $G$-orbits in a periodic orbit, but we start with a $G$-orbit that returns to a given equivariant disc, see Definition \ref{defi:equidisc} below, after a given 
time.

\begin{definition}\label{defi:equidisc} 
Let $M$ be a $G$-manifold, $\gamma\subseteq M$ a $G$-orbit and $U\cong G\times _HV$ a tubular neighbourhood for $Gx$, $H=G_x$. Assume that $\dim V^H>0$. Let $L$ be a $1$-dimensional 
subspace of $V^H$ and $L^\bot$ an invariant orthogonal complement in $V$. Let $p:V\to L$ be the orthogonal projection and $\pi:G\times_HV\to G\times_HL,\;[g,v]\mapsto[g,pv]$. The image of
the pullback of this bundle to the orbit $Gx$ in $U$ is called an equivariant disc in $M$, centered at $\gamma$. The image in $U$ of the pullback of a disc subbundle of $\pi$ is called an 
equivariant subdisc (of the disc defined above). 
\end{definition}

More generally, the above definition gives the notion of a codimension 1 equivariant disc, but we will not need greater generality. Under the identification $U\cong G\times_HV$, an 
equivariant disc centered at $[G, 0]$ can be identified with the subspace $G\times_H\setB(L^\bot)$, where $\setB(W)$ denotes the unit ball in $W$. 

Equivariant discs are the possible domain of definition for equivariant Poincar\'{e} maps. Let $\xi$ be an equivariant vector field, $\varphi$ its flow. Assume $\gamma$ is a $G$-orbit that 
is not fixed by the flow, i.e. there is a $t>0$ such that $\varphi(x, t)\notin\gamma$ for some $x\in\gamma$. Then there exists an $\eps>0$ such that $\varphi$ is an equivariant embedding of 
$\gamma\times\inoo{-\eps,\eps}$ into $M$. If $(H)$ is the orbit type of $\gamma$, in a suitable tubular neighbourhood $G\times_HV$ of $\gamma$ this can be interpreted as the embedding of 
the subspace $G\times_HL$ into $G\times_HV$, where $L\subseteq V^H$ is a one-dimensional subspace (corresponding to the direction of the flow). Let $D$ be the equivariant disc defined by 
$G\times_HV$ and $L$. Assume that $\gamma$ returns to $D$ in finite time. Since $\gamma$ is not fixed by the flow, there is a minimal $T>0$ such that $\varphi(\gamma, T)\subseteq D$ and 
in addition, $\xi(x)\notin T_xD$ for all $x\in D$. By continuity of $\varphi$, there is a subdisc $D'\subseteq D$ containing $\gamma$ and a continuous invariant function $t:D'\to\setR^+$, 
$t(\gamma)=T$, such that $\varphi(y,t(y))\in D$ for all $y\in D'$. By the implicit function theorem, $t$ is smooth and thus, the map
\[
 P:D'\to D,\;y\mapsto\varphi(y, t(y))
\]
is a smooth equivariant embedding. 

\begin{definition}
The collection $(D, D', P, t)$ is called an equivariant Poincar\'{e} system, centered at $\gamma$. $P$ is the equivariant Poincar\'{e} map.
\end{definition}

It is crucial that equivariant Poincar\'{e} systems are stable under small perturbations in the defining vector field.

\begin{proposition} 
Let $\xi:M\to TM$ be an equivariant vector field, $\varphi$ its flow, $\gamma\subseteq M$ a $G$-orbit that is not fixed by $\varphi$. Let $(D, D', P_0, t_0)$ be an equivariant Poincar\'{e}
system around $\gamma$. Then after possibly shrinking $D$, there is a neighbourhood $\calU$ of $\xi$ and a continuous map $t:\calU\to\calC^\infty_G(D',\setR^+)$ such that 
$(D, D', P(\eta), t(\eta))$ is a Poincar\'{e} system for $\eta\in\calU$, $P(\xi)=P$, $t(\xi)=t_0$, and $P(\eta)$ is defined in the obvious way.
\end{proposition}

\begin{proof} 
Since $\xi(x)\notin T_xD$, there is a neighbourhood of $\xi$ such that for all fields $\eta$ in this neighbourhood, $\eta(x)\notin T_xD$, after possibly passing to a subdisc. Hence, there 
is a small $\eps_x>0$ such that $\psi(x, t)\notin D$ for $0<t<\eps_x$ and $x\in D$ ($\psi$ the flow of $\eta$). So the return time map $t'$ is well-defined by $\psi(y, t'(y))\in D$ and 
$t'(\gamma)$ close to $T$. Continuous dependence is obvious by the implicit function theorem.
\end{proof}


\section{Transversality To Stratified Sets}

The theory of equivariant transversality is based on transversality to stratified sets, and we will therefore quickly review this theory. For proofs and a more detailed discussion, we once
more refer to \cite{field}, \cite{mather} and \cite{bochnak}. 

A stratified set is a topological space $X$ together with a decomposition $X=\bigcup_\alpha X_\alpha$ into a locally finite disjoint union of subspaces $X_\alpha$. The $X_\alpha$ are 
required to be smooth manifolds, usually non-compact.

It is apparent that in this generality, stratified spaces are not very well-behaved objects. We therefore assume firstly that the stratified space $X$ is a subset of a smooth manifold $M$.
Furthermore, we impose a condition on the stratification, which is known as Whitney's condition (b).

\begin{definition}
 A stratified space $X$ satisfies Whitney's condition (b), or is Whitney regular for short, if the following holds. Let $X_\alpha$, $X_\beta$ be two different strata of $X$ and
 $\{x_n\}_{n\in\setN}$, $\{y_n\}_{n\in\setN}$ be sequences in $X_\alpha$, $X_\beta$, respectively, converging to a point $x\in X_\alpha$. In a local coordinate system at $x$, let $L_n$ be 
 the line joining $x_n$ and $y_n$. Then, assuming that $L_n$ converges to a one-dimensional subspace $L$ of $T_xM$ and the spaces $T_{y_n}X_\beta$ converge to a subspace $E$ of
 $T_xM$, we must have $L\subseteq E$. 
\end{definition}

Whitney's condition (b) implies Whitney's condition (a), which states that, whenever a sequence $\{x_n\}_{n\in\setN}$ in a stratum $X_\beta$ converges to a point $x$ in $X_\alpha$ and the 
subspaces $T_{x_n}X_\beta$ converge to a subspace $E$ of $T_xM$, then $T_x{X_\alpha}\subseteq E$.

Transversality to a stratified set can be defined in a straightforward manner.

\begin{definition}
 Let $X$ be a stratified subspaces of the smooth manifold $N$, $M$ a smooth manifold. $f:M\to N$ is said to be transverse to $X$ at $x\in M$, if $f$ is transverse at $x$ to every stratum 
 of $X$, i.e. if $f(x)\in X_\alpha$, then
\[
 T_xf(T_xM)+T_{f(x)}X_\alpha=T_{f(x)}N.
\]
\end{definition}

The most important examples of stratified spaces are algebraic sets or more generally semialgebraic sets. A semialgebraic set is a set $Q\subseteq\setR^n$ such that there are polynomial maps
$p:\setR^n\to\setR^k$, $q:\setR^n\to\setR^m$ with
\[
 Q=\{x\in\setR^n\;|\;p(x)=0,\;q(x)>0\},
\]
where $q(x)>0$ stands for the same relation componentwise. Semialgebraic sets come equipped with a canonical stratification. We elaborate in the following what that means. Let 
$X\subseteq\setR^n$ be a subset and $S$ be a stratification of $X$, that is, $S$ is a collection of manifolds $X_\alpha$ constituting a stratification. For $i\in\setN$ we denote by $S_i$ 
the set of strata of $S$ of dimension $i$. Then a stratification $S$ is said to be minimal if, whenever $S'$ is a Whitney regular stratification, then there is an index $\ell\in\setN$ such
that $S_i=S_i'$ for $i>\ell$ and $S_\ell\supsetneq S_\ell'$. In this sense, a minimal stratification is the coarsest stratification of $X$. The following result is due to Mather, see 
Section 4 of \cite{mather}.

\begin{proposition}
 If a minimal stratification exists, then it is Whitney regular. Every semialgebraic set has a minimal stratification.
\end{proposition}

We will mostly deal with algebraic sets. Transversality to an algebraic set will always mean transversality to the minimal stratification. We also call the minimal stratification the 
canonical stratification.

It turns out, compare \cite{mather} or \cite{bochnak}, that Whitney's condition (a) ensures that the set of maps transverse to a stratification is open. Whitney's condition (b) is needed 
for density and further properties. We quote the final result from \cite{field}, Theorem 3.9.1.

\begin{theorem}\label{thm:thommatherstrat}
 Let $M, N$ be smooth compact manifolds and $X\subseteq N$ a closed Whitney stratified subset.
\begin{enumerate}
 \item The set $\calN$ of smooth maps $f:M\to N$ that are transverse to $X$ is open and dense in $\calC^\infty(M, N)$.

 \item If $H:M\times\incc{0,1}\to M$ is a smooth homotopy and $H_t$ is transverse to $X$ for every $t\in\incc{0,1}$, then there is a continuous isotopy $k:M\times\incc{0,1}\to M$ such that
 $k_0=\id_M$ and $k_t(H_t^{-1}(X))=H_0^{-1}(X)$.
\end{enumerate}
\end{theorem}

In view of bifurcation theory it is useful to look at maps dependend on a parameter. Hence, let $H:M\times\Lambda\to N$ be a smooth map, where $\Lambda$ is a smooth manifold as well.
This is of course just a special case of what has been done so far. But interpreting the space $\Lambda$ as a parameter space, we would like to know what it means for the various maps
$H_\lambda=H(\cdot, \lambda)$ to come from a map $H$ transverse to some set $X$. 

\begin{proposition}\label{prop:parastrat}
 Let $M, N, \Lambda$ be smooth manifolds and $X\subseteq N$ a closed Whitney stratified subset. Let $H:M\times\Lambda\to N$ be a smooth map transverse to $X$. If $M$ is compact, the set 
 of parameters $\{\lambda\in\Lambda\;|\;H_\lambda\mbox{ transverse to }X\}$ is open and dense in $\Lambda$.
\end{proposition}

\begin{proof}
 We will give a proof of this fact for $G$-transversality, compare Proposition \ref{prop:equiparatrans}. The proof itself is very similar to what is needed here, so we leave the details 
 to the interested reader.
\end{proof}

We need one more fact on Whitney regular stratifications, regarding the preimages under a transverse map. Classically, the preimage of a submanifold under a map transverse to that manifold
is a manifold itself. A similar result holds for maps transverse to a Whitney stratification. This is Corollary 8.8 of \cite{mather}

\begin{proposition}\label{prop:transpreimage}
 Let $M, N$ be smooth manifolds and $X\subseteq N$ a closed Whitney stratified subset. If $f$ is transverse to $X$, then the sets $f^{-1}(X_\alpha)$ constitute a Whitney regular 
 stratification of $f^{-1}(X)$.
\end{proposition}


\section{Equivariant Non-Degeneracy}
Non-equivariantly, non-degeneracy of a map $f:M\to M$ is defined as transversality of the map $(\id_M, f):M\to M\times M$ to the diagonal submanifold $\Delta=\Delta(M)\subseteq M\times M$.
Therefore it seems reasonable to define a notion of equivariant non-degeneracy similarly in terms of equivariant transversality. There are some slight pitfalls in this approach, but
as we will show, the notion is sufficient to prove an equivariant Thom-Mather Theorem for equivariantly non-degenerate maps. An interesting open problem is the question whether 
$G$-hyperbolicity implies equivariant non-degeneracy, which is almost trivial in the case with trivial group action.

We start by defining equivariant transversality as in \cite{field} or \cite{bierstone}, where it is also shown that the following definitions are independent of all choices. As usual, the 
definition splits into the essential part of defining transversality to $0$ in a representation, and the generalization to manifolds. It is essential for the theory that for 
$G$-representations $V, W$, the $\calC^\infty_G(V, \setR)$-module $\calC^\infty_G(V, W)$ is finitely generated and generators can be chosen as polynomial maps. This follows from an
equivariant Stone-Weierstrass Theorem, compare \cite{field}, Lemma 6.6.1, and the references therein.

Before we give the definition, we establish the general setting. Let $V, W$ be $G$-representations and $f:V\to W$ a $G$-map. We choose finitely many polynomial generators $F_1,\dots, F_k$ 
for the $\calC^\infty_G(V, \setR)$-module $\calC^\infty_G(V, W)$. Hence, there is a decomposition $f=\sum_{i=1}^kf_i\cdot F_i$ for $G$-maps $f_i:V\to\setR$. We write
\[
 \vartheta:V\times\setR^k\to W,\;(v, t)\mapsto\sum_{i=1}^kt_i\cdot F_i(v),
\]
the universal polynomial, and
\[
 \Gamma_f:V\to V\times\setR^k,\;v\mapsto(v, f_1(v), \dots, f_k(v)),
\]
the graph map associated with $f$. We have $f=\vartheta\circ\Gamma_f$. We define the universal variety $\Sigma=\Sigma(V, W)$ as the set
\[
 \Sigma=\vartheta^{-1}(0).
\]
Thus, $f(v)=0$ if and only if $\Gamma_f(v)\in\Sigma$.

\begin{definition}
 With the notation as above, a map $f:V\to W$ is said to be equivariantly transverse to $0$ at $0$ if $\Gamma_f$ is transverse to the stratified set $\Sigma$ at $0$.
\end{definition}

Next, we carry over this definition to the case of maps between $G$-manifolds. For this again, we need some background construction. Let $M, N$ be $G$-manifolds and $f:M\to N$ a smooth
$G$-map. We choose a slice $S_x$ at $x$ for the $G$-action and a $G_x$-invariant neighbourhood $U$ of $f(x)$ such that $f(S_x)\subseteq U$. We can achieve that there is a 
$G_x$-diffeomorphism
\[
 \zeta:U\to T_{f(x)}N
\]
such that $\zeta(f(x))=0$. Furthermore, $T_{f(x)}N$ splits as $T_{f(x)}P\oplus T_{f(x)}P^\bot$ and we can assume that $\zeta$ maps $U\cap P$ diffeomorphically onto $T_{f(x)}P$. The slice 
$S_x$ is $G_x$-diffeomorphic to $V_x=T_xGx^\bot$ via a $G_x$-diffeomorphism $\varphi:S_x\to V_x$, satisfying $\varphi(x)=0$. Let $\pi_2:T_{f(x)}P\oplus T_{f(x)}P^\bot$ be the projection
to the second factor. Then the map $F=\pi_2\circ\zeta\circ f\circ\varphi:V_x\to T_{f(x)}P^\bot$ has the property that $F(v)=0$ if and only if $f(\varphi^{-1}(v))\in P$.

\begin{definition}
Let $M, N$ be smooth $G$-manifolds and $f:M\to N$ a smooth $G$-map, $P\subseteq N$ a smooth invariant submanifold. With the notation from above, we define $f$ to be equivariantly 
transverse (or $G$-transverse) to $P$ at $x$, if either $f(x)\notin P$, or else $F$ is $G_x$-transverse to $0$ at $0$.  
\end{definition}

\begin{remark}
It is easily seen that the notion of equivariant transversality is invariant under group translations. That is, if $f:M\to N$ is $G$-transverse to the submanifold $P\subseteq N$ at $x$, 
then it is $G$-transverse to $P$ at $g.x$. 
\end{remark}

The following is the fundamental genericity theorem for equivariant transversality, compare \cite{field}, Theorem 6.14.1 and \cite{bierstone}, Proposition 6.5.

\begin{theorem}\label{thm:equithommather}
 Let $M, N$ be $G$-manifolds, $P\subseteq N$ an invariant submanifold.
 \begin{enumerate}
  \item If $P$ is closed, the set of smooth $G$-maps $f:M\to N$ $G$-transverse to $P$ at all of $M$ is open in the smooth Whitney topology. If $M$ is compact, it is open in the 
  $\calC^\infty$-topology.

  \item The set of smooth $G$-maps $f:M\to N$ $G$-transverse to $P$ at all of $M$ is the intersection of countably many open and dense subsets. In particular, it is dense, since 
  $\calC^\infty(M, N)$ is a Baire space.

  \item If $M$ is compact, $P$ is closed and $H:M\times\incc{0,1}\to M$ is a smooth $G$-homotopy, then there is a continuous equivariant isotopy $k:M\times\incc{0,1}\to M$ such that
  $k_0=\id_M$ and $k_t^{-1}(H_t^{-1}(P))=H_0^{-1}(P)$.

  \item If $f:M\to N$ is $G$-transverse to the closed invariant submanifold $P\subseteq N$, then $f^{-1}(P)$ has a minimal Whitney stratification.
 \end{enumerate}
\end{theorem}

We proceed to establish some elementary properties of equivariantly transverse maps. Since not all the results are to be found in the literature in the specific formulation we need,
we provide some of the proofs, where this is reasonable.

\begin{proposition}\label{prop:equiparatrans} Let $M, N, \Lambda$ be $G$-manifolds, $\Lambda$ with trivial $G$-action, $P\subseteq N$ an invariant submanifold. Let $F:\Lambda\times M\to N$
be a $G$-map $G$-transverse to $P$. Then the set of parameters $\lambda\in\Lambda$ such that $F_\lambda:M\to N$ is $G$-transverse to $P$ is residual in $\Lambda$. If $P$ in addition is 
compact, this set is open.
\end{proposition}
      
\begin{proof} Openness in case of compactness follows immediately from genericity of $G$-transverse maps and thus from Theorem \ref{thm:equithommather}. We proceed to show density, where 
we follow \cite{bierstone}. We know from Theorem \ref{thm:equithommather} (iv) that $F^{-1}(P)$ has a minimal Whitney stratification. Let $\pi:F^{-1}(P)\to\Lambda$ be the projection. We 
will show that $F_\lambda$ is $G$-transverse to $P$ if $\lambda$ is a regular value for the restrictions of $\pi$ to any stratum of $F^{-1}(P)$. Then Sards theorem takes care of the rest.
      
Clearly, if $(\lambda, x)\notin F^{-1}(P)$, then $F_\lambda$ is $G$-transverse to $P$ at $x$. So assume $F(\lambda, x)\in P$. Working locally, we can assume $M=V, N=W$ are 
$G$-representations, $\Lambda=\setR^\ell$, $F(\lambda, 0)=0$ and $F$ is $G$-transverse to $0$ at $(\lambda, 0)$. We have to check in which circumstances $F_\lambda$ is $G$-transverse to 
$0$ at $0$. Choose generators $F_1,\dots, F_k$ of $\calC^\infty_G(V, W)$ and write $F(\lambda, x)=\sum_{j=1}^kf_j(\lambda, x)\cdot F_j(x)$, which is possible since $G$ acts trivially on 
$\Lambda$. Let $\Sigma_\lambda$ be the stratum of $\Sigma_G(V, W)$ containing $(0, f(\lambda,0))$. Then $\Sigma=\Lambda\times\Sigma_\lambda$ is the stratum of 
$\Sigma_G(\Lambda\times V, W)$ containing $(\lambda, 0, f(\lambda,0))$. By assumption, the map $(\mu, x)\mapsto(\mu, x, f(\mu, x))$ is transverse to $\Sigma$ at $(\lambda, 0)$ and we have 
to check under which conditions $x\mapsto(x, f(\lambda, x))$ is transverse to $\Sigma_\lambda$ at $0$.
      
Assume $\lambda$ is a regular value for the restrictions of the projection $\pi:F^{-1}(0)\to\Lambda$ to the strata of $F^{-1}(0)$. Since 
$F^{-1}(0)=\Gamma_F^{-1}(\Sigma_G(\Lambda\times V,W))$, in particular the projection $p:\Gamma_F^{-1}(\Sigma)\to\Lambda$ has $\lambda$ as a regular value, so
\[
 T_{(\lambda,0)}(\Lambda\times V)=T_{(\lambda,0)}\Gamma_F^{-1}(\Sigma)+T_0V.
\]
By this and transversality of $\Gamma_F$ to $\Sigma$ at $(\lambda, 0)$, we have
\begin{eqnarray*}
T_{\Gamma_F(\lambda,\,0)}(\Lambda\times V\times\setR^k)&=&T_{(\lambda,0)}\Gamma_F(T_{(\lambda,0)}(\Lambda\times V))+T_{\Gamma_F(\lambda,\,0)}(\Sigma)\\
                                                       &=& T_0\Gamma_{F_\lambda}(T_0V)+T_{\Gamma_F(\lambda,0)}(\Sigma)\\
                                                       &=& T_0\Gamma_{F_\lambda}(T_0V)+T_\lambda\Lambda+T_{\Gamma_{F_\lambda}(0)}(\Sigma_\lambda).
\end{eqnarray*}
Since $T_{\Gamma_F(\lambda,\,0)}(\Lambda\times V\times\setR^k)=T_\lambda\Lambda+T_{\Gamma_{F_\lambda}(0)}(V\times\setR^k)$, we see that 
\[
 T_{\Gamma_{F_\lambda}(0)}(V\times\setR^k)=T_0\Gamma_{F_\lambda}(T_0V)+T_{\Gamma_{F_\lambda}(0)}\Sigma_\lambda
\]
which proves our claim.
\end{proof}

An important property of equivariant transversality is that we can draw conclusions on the various fixed point maps $f^H:M_{(H)}^H\to N^H$, considered as non-equivariant maps. More 
precisely it turns out that all these maps are transverse to the manifold $P^H$, provided $f$ is equivariantly transverse to $P$. This property is called stratumwise transversality of $f$.

\begin{proposition}\label{prop:stratumwisetrans}
 Let $M, N$ be $G$-manifolds, $P\subseteq N$ an invariant submanifold. Let $f:M\to N$ be a $G$-map $G$-transverse to $P$. Then for every isotropy subgroup $H$, the map
 \[
  f^H:M^H_{(H)}\to N^H
 \]
 induced by $f$ is non-equivariantly transverse to $P^H$.
\end{proposition}

\begin{proof} Again, we closely follow \cite{bierstone}. Working locally, we see that the set $M^H_{(H)}$ corresponds to the set $S_x^H$ in a normal slice at $x$. So 
$f^H:M^H_{(H)}\to N^H$ is transverse to $x\in P^H$ if and only if the map $\pi_2\circ\zeta\circ f\circ\varphi^{-1}:S_x^H\to\left(T_{f(x)}P^\bot\right)^H$ is transverse to $0$ at $0$. 
Consequently, the whole problem reduces to the following special case. Let $V, W$ be $G$-representations. If $f:V\to W$ is $G$-transverse to $0$ at $0$, then $f^G:V^G\to W^G$ is transverse
to $0$ at $0$. 
      
We write $V=V^G\oplus A$, $W=W^G\oplus B$ and define $F_i(v, a)=e_i$ for $i=1,\dots, \dim W^G=k$, where $e_i$ is a basis for $W^G$. Then we add generators $F_i(v, a)=\tilde{F}_i(a)$, 
$i=k+1,\dots, m$, where the $\tilde{F}$ are generators of $\calC^\infty_G(A, B)$. Then $F_1,\dots, F_m$ generates $\calC^\infty_G(V, W)$. For this special choice of generators, 
\[
 \Sigma_G(V, W)=\{(v, a, t)\in V^G\times A\times\setR^m\;|\; t=(s, 0)\in\setR^{m-k}\times\setR^k, (s, a)\in\Sigma_G(A, B)\}
\]
which we can identify with $V^G\times\Sigma_G(A, B)\times\{0\}\subseteq V^G\times(A\times\setR^{m-k})\times\setR^k$. The map $\pi\circ\Gamma_f\circ i$ is equal to $f^G$, where $i:V^G\to V$
is the inclusion, $\pi:V\times\setR^{m-k}\times\setR^k\to\setR^k$ the projection. Since $f$ is $G$-transverse to $0$ at $0$, we have
\[
 T_0\Gamma_f(V^G)+T_0\Gamma_f(A)+T_{\Gamma_f(0)}\Sigma_G(V, W)=V\times\setR^m.
\]
But by definition of the generators, $T_0\Gamma_f(A)$ is just $T_0\Gamma_{f\big|_A}(A)$ and since $f\big|_A$ is $G$-transverse to $0$ at $0$, the latter adds up with 
$T_{\Gamma_{f|_A}(0)}\Sigma_G(A, B)$ to $A\times\setR^{m-k}$. Furthermore,      
\[
 T_{\Gamma_f(0)}\Sigma_G(V, W)=V^G\times T_{\Gamma_{f|_A}(0)}\Sigma_G(A, B).
\]
Thus, $G$-transversality of $f$ implies that
\[
 T_0\Gamma_f(V^G)+V\times\setR^{m-k}=V\times\setR^m.
\]
Composition with $\pi$ gives
\[
 T_0f^G(V^G)+\{0\}=\setR^k=W^G.
\]
This is transversality of $f^G$ to $0$ at $0$.
\end{proof}

We proceed to define the notion of equivariant non-degeneracy. As pointed out above, this is basically just equivariant transversality of the map $(\id_M, f)$ to the diagonal submanifold, 
but we also add some extra generality by introducing parameters and dealing with maps out of submanifolds. 

\begin{definition}
 Let $M$ be a $G$-manifold, $P\subseteq M$ an invariant submanifold and $s$ a non-negative integer. Let $f:P\times\setR^s\to M$ be a $G$-map and $\pi_1:P\times\setR^s\to P$ be the projection
 onto the first factor. $f$ is called equivariantly non-degenerate at a point $(x, \lambda)\in P\times\setR^s$, if the map $(\pi_1,f):P\times\setR^s\to P\times M$ is equivariantly 
 transverse to the $P$-diagonal 
 \[
  \Delta=\Delta(P, M)=\{(x, x)\;|\;x\in P\}\subseteq P\times M
 \]
 at $(x, \lambda)$. 
\end{definition}

The main difficulty for the proof of an equivariant Thom-Mather Theorem for equivariant non-degeneracy is the density part. This comes from the fact that we would like to have a sequence
$(\id_M, f_n)$ of equivariantly transverse maps converging to $(\id_M, f)$. But the equivariant Thom-Mather Theorem \ref{thm:equithommather} only provides some sequence 
$(f^1_n, f^2_n)$ with that property, with $f^1_n$ not necessarily the identity. So we have to proceed differently.

The following lemma allows us to carry over genericity results from transversality to stratifications to equivariant transversality. It states that equivariant non-degeneracy is determined
not only pointwise but locally by the transversality property of a fixed graph map.

We need some preparations before the actual statement. Let $M$ be a compact $G$-manifold, $P\subseteq M$ an invariant compact submanifold and $s$ a non-negative integer. Let 
$f:P\times\setR^s\to M$ be a $G$-map and $f(x, \lambda)=x$, $G_x=H$. Let $S_x\subseteq P$ be a slice at $x$, $U\subseteq M$ be an $H$-neighbourhood of $x$ such that 
$f(S_x\times J)\subseteq U$ for a neighbourhood $J$ of $\lambda$. Let $\zeta:U\to T_xM$ be an $H$-diffeomorphism with $\zeta(x)=0$ and $\zeta(S_x)=T_xGx^\bot=V$. We write 
$\zeta\big|_{S_x}=\varphi$. By definition of equivariant non-degeneracy, $f$ is equivariantly non-degenerate at $(x, \lambda)$ if and only if the map
\[
 F=\pi\circ(\zeta\times\zeta)\circ(\pi_1, f)\circ(\varphi^{-1}, \id):V\times J\to T_{(x, x)}\Delta(M)^\bot
\]
is $H$-transverse to $0$ at $0$, where $\pi:T_xM\times T_xM\to T_{(x, x)}\Delta(M)^\bot=W$ is the projection. Let $F_1,\dots, F_k$ be finitely many polynomial generators for 
$\calC^\infty_H(V, W)$,
\[
 \vartheta:V\times\setR^k\to W,\;(v, t)\mapsto\sum_{i=1}^kt_i\cdot F_i(v)
\]
the associated universal polynomial,
\[
 \Gamma_F:V\times J\to V\times\setR^k,\;(v, \mu)\mapsto(v, f_1(v, \mu),\dots, f_k(v, \mu))
\]
the graph map of $F$.

\begin{lemma}\label{lemma:equifromstrat}
With the setup as above, $f$ is equivariantly non-degenerate at $(y, \mu)\in S_x\times J$ if and only if $\Gamma_F$ is transverse to the universal variety $\vartheta^{-1}(0)$ at 
$(\varphi(y), \mu)$.
\end{lemma}

\begin{proof}
 We only have to check the statement for points $(y, \mu)\in S_x\times J$ such that $f(y, \mu)=y$. Let $v=\varphi(y)$. A small ball $S_v\subseteq T_vHv^\bot\subseteq V$ around $v$ can 
 serve as a normal slice for the $H$-action on $V$ at $v$, and the image $S=\varphi^{-1}(S_v)$ is a normal slice for the $G$-action at $y$. Let $K=G_y$ and $Z=S_v-\{v\}$, a ball around $0$ 
 in a $K$-representation. We define
 \[
  \zeta':U\to T_yM,\;\zeta'(z)=\zeta(z)-v.
 \]
 Accordingly, $\varphi'$ is given as $\varphi'(z)=\varphi(z)-v$. Then equivariant non-degeneracy of $f$ at $(y, \mu)$ is equivalent to $K$-transversality of
 \[
  F'=\pi\circ(\zeta'\times\zeta')\circ(\pi_1, f)\circ(\varphi'^{-1}, \id):Z\times J\to W.
 \]
 By definition of $\zeta'$, $F'(z, \mu)=F(z+v, \mu)$. We need to choose generators for $\calC^\infty_K(Z, W)$ and we can take $F'_i(z)=F_i(z+v)$. The universal variety is then given as
\begin{eqnarray*}
 \Sigma'=\vartheta'^{-1}(0)&=&\{(z, t)\in Z\times\setR^k\;|\;\sum_{i=1}^kt_i\cdot F_i(z+v)=0\}\\
                           &=&\{(z, t)\in Z\times\setR^k\;|\;(z+v, t)\in\Sigma\}\\
                           &=&(\Sigma-\{(v, 0)\})\cap(Z\times\setR^k)\\
                           &=&(\Sigma\cap(S_v\times\setR^k))-\{(v,0)\}
\end{eqnarray*}
Furthermore,
\[
 F'(z, \mu)=F(z+v, \mu)=\sum_{i=1}^kf_i(z+v, \mu)\cdot F_i(z+v)=\sum_{i=1}^kf_i(z+v, \mu)\cdot F'_i(z),
\]
and hence,
\[
 \Gamma_{F'}(z)=(z, f_1(z+v, \mu), \dots, f_k(z+v, \mu))=\Gamma_F(z+v)-(v, 0).
\]
By definition, $f$ is equivariantly non-degenerate at $(y, \mu)$ if and only if $\Gamma_{F'}$ is transverse to $\Sigma'$ at $(0, \mu)$, which is equivalent by our calculations to 
$\Gamma_F\big|_{S_v}$ being transverse to $\Sigma\cap(S_v\times\setR^k)$ at $(v, \mu)$. But $S_v$ is a slice for the $H$-action, hence by translation invariance of equivariant 
transversality, $\Gamma_F\big|_{S_v}$ is transverse to $\Sigma\cap(S_v\times\setR^k)$ at $(v, \mu)$ if and only if $\Gamma_F$ is transverse to $\Sigma$ at $(v, \mu)$. This proves our 
claim.
\end{proof}

\begin{theorem}\label{thm:equinondegdense}
 Let $M$ be a compact $G$-manifold and $P\subseteq M$ a compact invariant submanifold. Then the space of $G$-maps $P\times\setR^s\to M$ that are equivariantly non-degenerate in a compact 
 subset $K\subseteq P\times\setR^s$ is open and dense.
\end{theorem}

\begin{proof}
 Openness is immediate from Theorem \ref{thm:equithommather}. For density, let $f:P\times\setR^s\to M$ be any $G$-map and $(x, \lambda)$ be a fixed point of $f$ in $K$, i.e. 
 $f(x, \lambda)=x$. We choose a slice $S_x\subseteq P$ at $x$ and a $G_x$-diffeomorphism $\varphi:S_x\to V=T_xGx^\bot$ such that $\varphi(x)=0$. Furthermore, we choose a 
 $G_x$-neighbourhood $U$ of $x$, a $G_x$-diffeomorphism $\zeta:U\to T_xM$ and a neighbourhood $J\subseteq\setR^s$ of $\lambda$ such that $\zeta(x)=0$ and $f(S_x, J)\subseteq U$. We can 
 choose $\varphi$ and $\zeta$ in way such that $\zeta\circ\varphi^{-1}:V\to T_xM$ is the inclusion, and we do so.

 Let $\pi:T_xM\times T_xM\to T_xM,\;(v, w)\mapsto v-w$. $\pi$ can be interpreted as the projection onto the orthogonal complement of the diagonal subspace of $T_xM\times T_xM$. In particular,
 $G$-transversality of $(\id_M, f)$ is defined as $G_x$-transversality of the map
\[
 F:V\times J\to T_xM
\]
\[
(v, \lambda)\mapsto\pi\circ\zeta\times\zeta(\varphi^{-1}(v), f(\varphi^{-1}(v), \lambda))=\pi(v, \zeta\circ f(\varphi^{-1}(v), \lambda))=v-\zeta\circ f(\varphi^{-1}(v), \lambda))
\]
to $0$ at $0$. Let $F_1,\dots, F_k$ be a finite set of polynomial generators for $\calC^\infty_H(V, T_xM)$. Writing
\[
 F(v, \lambda)=\sum_{i=1}^kf_i(v, \lambda)\cdot F_i(v),
\]
we have the graph map
\[
 \Gamma_F:V\times J\to V\times\setR^k,\;(v, \lambda)\mapsto(v, f_1(v, \lambda), \dots, f_k(v, \lambda)),
\]
and transversality of $\Gamma_F$ to the universal variety $\Sigma$ at $0$ is equivalent to $G$-trans\-ver\-sa\-li\-ty of $(\id_M, f)$ to $\Delta$ at $(x, \lambda)$. Define a map
\[
 \Gamma:V\times J\times\setR^k\to V\times\setR^k,\;(v, \lambda, b)\mapsto(v, f_1(v, \lambda)+b_1,\dots, f_k(v, \lambda)+b_k).
\]
Then $\Gamma$ is transverse to the universal variety at all of $V\times J\times\setR^k$. By Proposition \ref{prop:equiparatrans}, there are arbitrarily small parameters $b$ such that 
$\Gamma^b=\Gamma(\cdot, \cdot, b)$ is transverse to $\Sigma$ at all of $V\times J$. For any parameter $b$, we obtain a $G_x$-map $F^b:V\times J\to T_xM$ via
\[
 F^b(v, \lambda)=\vartheta\circ\Gamma^b(v, \lambda)
\]
and this in turn determines a $G_x$-map 
\[
f^b:S_x\times J\to U, (y, \mu)\mapsto\zeta^{-1}(\varphi(y)-F^b(\varphi(y), \mu)).
\]
By choosing $b$ sufficiently small, this map $f^b$ is arbitrarily close to $f\big|_{S_x}$. By Lemma \ref{lemma:equifromstrat}, the unique $G$-map $G.S_x\times J\to M$ resulting from $f^b$
is equivariantly non-degenerate, since its graph map is $\Gamma^b$.

Now assume we are given a neighbourhood $\calU$ of $f$. For every orbit $Gx$ of fixed points of $f$, we choose a slice $S_x$ and neighbourhoods $J_x, U_x$ as above. The sets 
$G.S_x\times J_x$ cover $\fix(f)$ and by compactness, we find finitely many, centered around points $(x_1, \lambda_1),\dots, (x_m, \lambda_m)\in\fix(f)$. Denoting the corresponding sets by 
$S_i$, $J_i$, $U_i$ for $i=1,\dots, m$, we then replace $f$ in $S_1\times J_1$ by a map $f^b$, constructed in the way sketched above, and extend uniquely to $G.S_1\times J_1$ by 
equivariance.

Clearly, this replacement can be smoothed to a $G$-map $f_1\in\calU$ such that $f_1$ is equivariantly non-degenerate in a neighbourhood of $G.S_1\times J_1$ by choosing $b$ sufficiently 
small. Now assume we have a map $f_k\in\calU$ which is equivariantly non-degenerate in the closure of $\bigcup_{i=1}^k(G.S_i\times J_i)$. By openness of equivariant non-degeneracy, we can 
then alter $f_k$ in $S_{k+1}\times J_{k+1}$ to a map $f_k^b$ such that the resulting map remains equivariantly non-degenerate in $\bigcup_{i=1}^k(G.S_i\times J_i)$. Inductively, we 
obtain a map $f'\in\calU$ which is equivariantly non-degenerate in a neighbourhood of the set of fixed points of $f$. But given any neighbourhood $U$ of $\fix(f)$, there is a neighbourhood 
$\calU'$ of $f$ such that no element of $\calU'$ has a fixed point outside of $U$. Hence, the whole construction can be carried out in a way such that $f'$ has no fixed points outside of 
$\bigcup_{i=1}^m(G.S_i\times J_i)$ and thus is equivariantly non-degenerate in $K$.  
\end{proof}

We can now transfer the properties of equivariant transversality to equivariant non-degeneracy.

\begin{proposition}\label{prop:equiparatransnondeg}
 Let $M$ be a compact $G$-manifold, $P\subseteq M$ a compact invariant submanifold and $f:P\times\setR^s\to M$ a $G$-map that is equivariantly non-degenerate in $P\times I$ for some 
 subset $I\subseteq\setR^s$. Then the set of parameters $\lambda\in I$ such that $f_\lambda:P\to M$ is equivariantly non-degenerate is open and dense.
\end{proposition}

\begin{proof}
 This follows from Proposition \ref{prop:equiparatrans}, applied to the diagonal submanifold.
\end{proof}

\begin{proposition}\label{prop:equidimnondeg}
 Let $M$ be a $G$-manifold, $P$ an invariant submanifold, $s$ a non-negative integer and $f:P\times\setR^s\to M$ an equivariantly non-degenerate $G$-map. Then the set of fixed points of 
 $f$ of isotropy $H$ is either empty, or a submanifold of $P^H_{(H)}\times\setR^s$ of dimension $s-\dim M^H+\dim P^H$. If this value is less than zero, it follows in particular that $f$ 
 has no fixed points with isotropy $H$.
\end{proposition}

\begin{proof}
 Clearly we can assume that $P^H_{(H)}$ is non-empty, otherwise, there are no fixed points of isotropy $H$. Since the map $(\id_M, f)$ is equivariantly transverse to the diagonal, it is 
 stratumwise transverse to the diagonal by Proposition \ref{prop:stratumwisetrans}. Hence, the preimages $(\id_M,f^H)^{-1}(\Delta(P^H, M^H))$ are either empty, or submanifolds of 
 $P^H_{(H)}\times\setR^s$ of dimension 
\[
 \dim P^{H}_{(H)}+s-\dim P^H-\dim M^H+\dim P^H=s-\dim M^H+\dim P^H.
\]
 The last equality follows since $H$ is an actual isotropy in $P$ and therefore, $(H)$ is the principal orbit type in $P^H_{(H)}$, implying that $\dim P^H_{(H)}=\dim P^H$.
\end{proof}

\begin{remark}
 Looking closer at the cases $s=0$ and $s=1$, we remark the following. The set of fixed points of a $G$-map $f:M\to M$ is generically a finite union of $G$-orbits. If $G$ is infinite, only 
 $G$-orbits of type $(H)$ with $\dim W(H)=0$ participate.

 For $s=1$, the set of fixed points is generically a $1$-dimensional submanifold consisting of two parts. A finite union of $G$-orbits of type $(H)$ with $\dim W(H)=1$ at finitely many 
 parameters, and a 1-dimensional manifold consisting of finite numbers of $G$-orbits of type $(H)$ with $\dim W(H)=0$, parametrized smoothly by the $\setR$ component.
\end{remark}

\begin{corollary}\label{cor:equicodim}
 Let $M$ be a $G$-manifold and $P\subseteq M$ a compact invariant submanifold such that
\[
 \dim P^H_{(H)}+s<\dim M^H
\]
for an orbit type $(H)$ of $P$. Then the set of $G$-maps $M\times\setR^s\to M$ that have no fixed points of type $(H)$ in $P\times\setR^s$ is open and dense.
\end{corollary}
\begin{proof}
 This is an immediate consequence of Proposition \ref{prop:equidimnondeg}, together with the Density Theorem \ref{thm:equinondegdense}.
\end{proof}

\begin{corollary}\label{cor:equicodimdisc}
 Let $P$ be the finite disjoint union of boundaries of equivariant discs. Then the set of equivariant homotopies $M\times I\to M$ without fixed points in $P$
 is open and dense.
\end{corollary}

\begin{proof}
 By Corollary \ref{cor:equicodim}, it suffices to establish the estimate given there for equivariant discs. Hence we assume that $P$ is the boundary of an equivariant disc, centered at an orbit
 of type $(H)$. Since all orbits in $P$ have type lesser or equal than $(H)$, we only need to check orbit types $(K)\leq(H)$. Discs are defined in tubular neighbourhoods, so we can assume 
 $M=G\times_HV$ for some $H$-representation $V=W\oplus V^H$ with $\dim V^H\geq1$. The boundary of the equivariant disc is then given by $G\times_H S(L^\bot)$, where $L^\bot$ is a 
 codimension $1$ subspace of $V$ with orthogonal complement contained in $V^H$. Now we can estimate
 \begin{eqnarray*}
  \dim P^K_{(K)}+1-\dim M^K&=&\dim(G\times_HS(L^\bot))^K_{(K)}+1-\dim(G\times_HV)^K\\
                           &=&\dim(G\times_HS(L^\bot_{(K)})^K+1-\dim(G\times_HW\times V^H)^K\\
                           &=&\dim(G\times_HS(W_{(K)}))^K+1-\dim(G\times_HW)^K\\
                           & &-\dim V^H\\
                           &\leq&\dim(G\times_HS(W_{(K)}))^K-\dim(G\times_HW)^K\\
                           &\leq&\dim(G\times_HS(W))^K-\dim(G\times_HW)^K<0.
 \end{eqnarray*}
 For the last estimate, we have used that $\dim W^K\geq1$, since otherwise, $P^K_{(K)}$ would be empty and there would be nothing to prove. We conclude that equivariant homotopies 
 $M\times I\to M$ in an open and dense subset have no fixed points in $P$ of type $(K)$. Since there are only finitely many orbit types in $P$, the claim is proven.
\end{proof}

We summarize the results of this section in a Thom-Mather theorem for equivariant non-degeneracy.

\begin{theorem}\label{thm:equithommathernondeg}
 Let $M$ be a smooth compact $G$-manifold, $P\subseteq M$ an invariant compact submanifold. Let $K$ be a compact subset of $P\times\setR^s$.
\begin{enumerate}
 \item The set $\calN_K$ of smooth $G$-maps $P\times\setR^s\to M$ that are equivariantly non-degenerate in $K$ is open and dense.

 \item If $s=1$, $K=P\times\incc{0,1}$ and $H:P\times\setR^s\to M$ satisfies the condition that $H_t$ is equivariantly non-degenerate for all parameters $t\in\incc{0,1}$, then there is a 
 continuous equivariant isotopy $k:P\times\incc{0,1}\to P$ such that $k_0=\id_P$ and $k_t(\fix(H_t))=\fix(H_0)$.
\end{enumerate} 
\end{theorem}

\begin{proof}
 Part (i) has been proven as Theorem \ref{thm:equinondegdense}, part (ii) follows immediately from the definition of equivariant non-degeneracy and Theorem \ref{thm:equithommather}.
\end{proof}

\begin{remark}
We point out that we have established the following bifurcation picture for $G$-maps. Generically, a $G$-map has finitely many fixed points such that for their isotropy groups, we have 
$\dim W(H)=0$ (by Proposition \ref{prop:equidimnondeg}). A homotopy of $G$-maps has finitely many bifurcation parameters (by Proposition \ref{prop:equiparatransnondeg}) and finitely many 
branches of fixed orbits undergo finitely many bifurcations (again by \ref{prop:equidimnondeg} and by Theorem \ref{thm:equithommathernondeg} (ii)). Along these branches, again 
$\dim W(H)=0$ holds.
\end{remark}

As a conclusion to this chapter, we give some simple examples of calculations of equivariant non-degeneracy.

\begin{example}
 \begin{enumerate}
  \item Let $G=\setZ_2$ and $V$ be the standard $1$-dimensional representation of $\setZ_2$. A $G$-map $f:V\times\setR^s\to V$ has the form $f(v, \lambda)=v\cdot h(v, \lambda)$ for an 
  invariant map $h:V\times\setR^s\to\setR$. $(v, \lambda)$ is a fixed point of $f$ if either $v=0$, or else $h(v, \lambda)=1$. If $v\neq0$,
  then equivariant non-degeneracy is just ordinary non-degeneracy. So we assume $v=0$, implying that $G_v=\setZ_2$. Equivariant non-degeneracy is then defined as equivariant transversality
  of the map $V\times\setR^s\to V\times V,\;(v, \lambda)\mapsto(v, v\cdot h(v, \lambda))$.

  After projecting to the orthogonal complement of the diagonal, we have to check $\setZ_2$-transversality of the map $v\mapsto v\cdot(1-h(v, \lambda))$ to $0$ at $0$.
  As generator for $\calC^\infty_{\setZ_2}(V, V)$ we choose the identity. The universal variety is thus given as
\[
 \Sigma=\{(v, t)\;|\;t\cdot v=0\},
\]
  which is stratified by $\{0\}$ and $\Sigma\setminus\{0\}$. The graph map of $f$ is
\[
 \Gamma_f(v, \lambda)=(v, 1-h(v, \lambda)).
\]
  If $h(0,\lambda)=1$, then $\Gamma_f(0, \lambda)=(0,0)$ and $\Gamma_f$ is transverse to $\{0\}$ at $0$ if and only if $T_{(0, \lambda)}\Gamma_f$ is surjective. This is the case if and 
  only if $\partial_\lambda h(0, \lambda)\neq0$.

  If $h(0, \lambda)\neq1$, then $\Gamma_f(0, \lambda)\in\Sigma\setminus\{0\}$. The tangential space of $\Sigma$ in this case is the $y$-axis and therefore, $f$ is equivariantly 
  non-degenerate without further conditions. 

  In conclusion, $f$ is equivariantly non-degenerate at $(0, \lambda)$ if either $h(0, \lambda)$ equals $1$ and $\partial_\lambda h(0, \lambda)\neq0$, or else $h(0, \lambda)\neq1$.

  Ordinary non-degeneracy of $f$ at $(0, \lambda)$ would require $\partial_\lambda f(0, \lambda)\neq0$ or $h(0, \lambda)\neq 1$. Since $\partial_\lambda f(0, \lambda)=0$ necessarily,
  equivariant non-degeneracy detects the maps with $h(0, \lambda)=1$ and $\partial_\lambda h(0, \lambda)\neq0$ as generic. 

  \item Let $G=\setS^1$ and $V$ the standard $2$-dimensional representation. Again, the interesting fixed points of an $\setS^1$-map $f:V\times\setR^s\to V$ are the points $(0, \lambda)$.
  In that case, we have to check equivariant transversality of the $\setS^1$-map
  \[
   F:V\times\setR^s\to V,\;(x, y, \lambda)\mapsto(x, y)-f(x, y, \lambda).
  \]
  As generators for $\calC^\infty_{\setS^1}(V, V)$ we choose
\[
 F_1(x, y)=(x, y),\;F_2(x, y)=(-y, x).
\]
  Writing $f(x, y, \lambda)=a(x, y, \lambda)\cdot(x, y)+b(x, y, \lambda)\cdot(-y, x)$, we have
\[
 \Gamma_F(x, y, \lambda)=(x, y, 1-a(x, y, \lambda), -b(x, y, \lambda)).
\]
The universal variety $\Sigma$ is given by the equations $sx-ty=0$, $sy+tx=0$, hence, this is just $\{(0,0)\}\times\setR^2\cup\setR^2\times\{(0,0)\}$, stratified by $\{(0,0,0,0)\}$ and
$\Sigma\setminus\{(0,0,0,0)\}$.

Again we have two possibilities. We assume first that $a(0,0,\lambda)=1$ and $b(0,0,\lambda)=0$. In that case, $\Gamma_F(0,0,\lambda)=(0,0,0,0)$ and $\Gamma_F$ is transverse to zero if 
and only if $T_{(0,0,\lambda)}\Gamma_F$ is surjective. This is the case if and only if the matrix 
\[
(\partial_\lambda a(0,0,\lambda), \partial_\lambda b(0,0,\lambda))^T
\]
has rank 2, i.e. there are $1\leq i, j\leq s$ such that 
\[
 \partial_{\lambda_i}a(0,0,\lambda)\cdot\partial_{\lambda_j}b(0,0,\lambda)-\partial_{\lambda_j}a(0,0,\lambda)\cdot\partial_{\lambda_i}b(0,0,\lambda)\neq0.
\]
In particular this is never the case if $s\leq1$.

The other possibility is that $a(0,0,\lambda)\neq1$ or $b(0,0,\lambda)\neq 0$. In that case, $\Gamma_F(0,0,\lambda)\in\{(0,0)\}\times\setR^2$ and $\Gamma_F$ is transverse to $\Sigma$
without further conditions.

In conclusion, $(0,0,\lambda)$ is a non-degenerate fixed point if either $T_{(0,0,\lambda)}f$ is not the identity, or else, $s\geq2$ and the matrix 
\[
 \begin{pmatrix}
  \partial_\lambda\partial_xf_1(0,0,\lambda)\\
  \partial_\lambda\partial_xf_2(0,0,\lambda)
 \end{pmatrix}
\]
is regular.
\end{enumerate}
\end{example}


\section{Genericity of Non-Degenerate Vector Fields}

We proceed to the investigation of equivariant vector fields and establish the following setting. Let $\Omega\subseteq M$ be an open subset and let $a, b$ be two real numbers, $0<a<b$. 
The set of (parametrized) $G$-vector fields $\xi:M\times\setR^s\to TM$ such that there is no periodic point of $\xi$ on $\partial(\Omega\times\inoo{a, b})$ is denoted with 
$\frakX_G(M\times\setR^s, \Omega, a, b)$. The periodic orbits of an element $\xi\in\frakX_G(M\times\setR^s, \Omega, a, b)$ lying inside of $\Omega\times\incc{a, b}$ are called the 
essential periodic orbits, otherwise they are called inessential. Note that the condition to be an element of $\frakX_G(M\times\setR^s, \Omega, a, b)$ means that there are no periodic orbits of a 
period in $\incc{a, b}$ meeting the boundary of $\Omega$, and there are no periodic orbits in $\Omega$ of period $a$ or $b$. The notion of equivariant non-degeneracy of a vector field is 
the following.

\begin{definition}
 Let $\xi\in\frakX_G(M\times\setR^s,\Omega, a, b)$ be a vector field, $\varphi:M\times\setR\times\setR^s$ its flow. Then $\xi$ is called equivariantly non-degenerate at a point 
 $(x, t, \lambda)\in\Omega\times\inoo{a,b}\times\setR^s$ provided its flow is equivariantly non-degenerate at that point.
\end{definition}

The proof of density of equivariantly non-degenerate $G$-vector fields depends on a geometric construction which stems from \cite{field1}. The techniques of the preceeding section
seem not be be adjustable to this case, since the additional condition of the maps under consideration to be flows prevents any straight forward adaption.

\begin{lemma}\label{lemma:equipoincarehomotopy} 
Let $H$ be an element of $\frakX_G(M\times\setR, \Omega, a, b)$ and $\gamma_\lambda$ be an essential periodic orbit of $H_\lambda$. Choose a Poincar\'{e} system 
$(D, D', P_\lambda, t_\lambda)$ for $\gamma_\lambda$ such that the Poincar\'{e} maps of all fields in a neighbourhood $\calU_1$ of $H_\lambda$ are defined as maps $D'\to D$. Let $V, U$ be 
invariant open neighbourhoods of $\gamma_\lambda$ such that
\[
 \conj{U}\subseteq\bigcup_{x\in D'}\varphi_\mu(\incc{0, t_\mu(x)})
\]
and
\[
 \varphi_\mu(\incc{0, t_\mu(x)})\subseteq U
\]
for $x\in\conj{V}\cap D'$ and $\mu$ in a neighbourhood of $\lambda$, say, if $\abs{\lambda-\mu}<3\eps$, $\eps>0$. Let $P$ be the homotopy
\[
 P:D'\times\incc{\lambda-3\eps,\lambda+3\eps}\to D
\]
given by the Poincar\'{e} maps of the $H_\mu$, $\mu\in\incc{\lambda-3\eps,\lambda+3\eps}$. Then there is a neighbourhood $\calU$ in the set of $G$-homotopies $D'\times\incc{\lambda-3
\eps,\lambda+3\eps}\to D$ equal to $P$ outside of $\conj{V}\cap D\times\incc{\lambda-2\eps,\lambda+2\eps}$ and a continuous map $\chi:\calU\to\frakX_G(M\times\setR, \Omega, a, b)$ such that
\begin{enumerate}

   \item for $Q\in\calU$, $\chi(Q)_\mu$ has Poincar\'{e} map $Q_\mu$ for $\abs{\lambda-\mu}<\eps$.
   
   \item for $Q\in\calU$, $\chi(Q)$ equals $H$ outside of $U\times\incc{\lambda-2\eps,\lambda+2\eps}$.
   
   \item $\chi(P)=H$.
   
\end{enumerate}
\end{lemma}

\begin{proof} 
Let $t_\xi$ be the period map of an element $\xi$ of $\calU_1$, i.e. $t_\xi$ is the minimal $t\in\setR^{>0}$ such that $\varphi_\xi(x, t_\xi(x))\in D$. Since $\gamma_\lambda$ is a periodic 
orbit and by definition of an equivariant Poincar\'{e} system, we have
\[
 t_0=\inf_{x\in D'}\inf_{\xi\in\calU_1}t_\xi(x)>0.
\]
Choose real numbers $r,s$ with $0<r<s<t_0$. Let $Q:D'\times\incc{\lambda-3\eps,\lambda+3\eps}\to D$ be any $G$-homotopy equal to $P$ outside of $V\cap D'\times\incc{\lambda-2\eps,\lambda+2\eps}$. By 
choosing $Q$ close enough to $P$, we can assume that $P^{-1}\circ Q$ (where $P^{-1}$ is to be taken fibrewise) is a $G$-embedding that is equivariantly isotopic to the inclusion $\conj{D'}\times
\incc{\lambda-3\eps,\lambda+3\eps}\hookrightarrow D\times\incc{\lambda-3\eps,\lambda+3\eps}$. 
Let
\[
 K:\conj{D'}\times\incc{\lambda-3\eps,\lambda+3\eps}\times\incc{r,s}\to D\times\incc{\lambda-3\eps,\lambda+3\eps}
\]
be an isotopy joining the two maps. For the flow $\varphi$ of $H$, define
\[
 \psi(y, \mu, t)=\varphi(K(y, \mu, t), t)
\]
for $y\in D'$, $\mu\in\incc{\lambda-3\eps,\lambda+3\eps}$, $t\in\incc{0, t_\mu(y)}$, where we extend $K$ smoothly to the interval $\incc{0,t_0}$ such that $K_0=K_r$, $K_s=K_{t_0}$ and 
every $K_t$ is equal to the inclusion close to the boundary of $D'\times\inoo{\lambda-3\eps,\lambda+3\eps}$. We have $\psi\equiv\varphi$ in a neighbourhood of the boundary of 
$D'\times\incc{\lambda-3\eps,\lambda+3\eps}$. By choosing $Q$ sufficiently close to $P$, we can achieve that none of the curves $t\mapsto\psi(y, \mu, t)$ meet $M - U$. Now $\varphi_\mu$ 
is an embedding when restricted to the interval $\incc{r,s}$ and the set of embeddings is open. Thus, $\psi_\mu$ can be made an embedding, too. Take 
$\mu\in\incc{\lambda-\eps,\lambda+\eps}$ and define
\[
 \eta_\mu(\psi(y, \mu, t))=\diff{s}\psi(y, \mu, s)\big|_{s=t}.
\]
$\eta_\mu$ is a $G$-vector field defined on the image of $\psi_\mu$. Extend $\eta$ to a homotopy on $M\times\incc{\lambda-3\eps,\lambda+3\eps}$ such that $\eta_\mu=H_\mu$ outside of 
$U\times\incc{\lambda-2\eps,\lambda+2\eps}$. Clearly, this extension can be done continuously in $Q$. The integral curves of $\eta$ for $\abs{\lambda-\mu}<\eps$ are, up to 
reparametrization, just the curves $t\mapsto \psi(y, \mu, t)$ and we calculate
\[
 \psi(y, \mu, t_\mu(y))=\varphi(K(y, \mu, t_\mu(y)), \mu, t_\mu(y))=\varphi(P_\mu^{-1}\circ Q_\mu(y), \mu, t_\mu(y))=Q_\mu(y).
\]
So the Poincar\'{e} homotopy induced by $\eta$ in $\incc{\lambda-\eps,\lambda+\eps}$ is given by $Q$. The definition $\chi(Q)=\eta$ gives the desired result.
\end{proof} 

Next we prove openness of equivariantly non-degenerate homotopies and a slightly more general result we will need in the proof of the density theorem.

\begin{proposition} 
Let $H\in\frakX_G(M\times\setR,\Omega, a, b)$ and $H$ is equivariantly non-de\-ge\-ne\-ra\-te in the compact subset $K\times J\subseteq\Omega\times I\subseteq M\times\setR$. Then there 
is a neighbourhood $\calU$ of $H$ such that every element of $\calU$ is equivariantly non-degenerate in $K\times J$. 
\end{proposition}

\begin{proof} 
Since $H$ is equivariantly non-degenerate, the map
\[
 F:M\times\setR\times\incc{a,b}\lra M\times M,\quad(x, t, \lambda)\longmapsto(x, \varphi_\lambda(x, t))
\]
is $G$-transverse to the diagonal in $K\times J\times\incc{a,b}$, $\varphi$ the flow of $H$. By openness of $G$-transversality, we find a neighbourhood $\calU_1$ of $F$ such that all 
elements of $\calU_1$ are $G$-transverse to the diagonal in $K\times J\times\incc{a,b}$. Clearly, if $H'$ is close to $H$, the associated map $F'$ is close to $F$, so we find a 
neighbourhood $\calU$ of $H$ of equivariantly non-degenerate maps in $K\times J\times\incc{a,b}$, i.e. all essential periodic orbits in $K\times J$ are equivariantly non-degenerate.
\end{proof}

With these two auxiliary results at hand, we can prove genericity of equivariantly non-degenerate homotopies of vector fields. It is based on the proof of the non-equivariant Kupka-Smale
theorem as can be found in \cite{palis}.

\begin{theorem}\label{thm:equigenvect}
The subset of $\frakX_G(M\times\incc{0,1}, \Omega, a, b)$ consisting of equivariantly non-degene\-rate $G$-homotopies is open and dense.
\end{theorem}

\begin{proof} 
Openness is a trivial corollary of the preceeding proposition. It therefore remains to prove density. We proceed in 5 steps.
\begin{enumerate}

   \item Take a homotopy $H\in\frakX_G(M\times\setR, \Omega, a, b)$ and let $\calU$ be any neighbourhood of $H$. Let $\Gamma\subseteq\Omega\times\inoo{a,b}$ be the set of essential geometric periodic 
   points of $H$, i.e. $(x, \lambda)\in\Gamma$ if and only if there is a $t\in\incc{a,b}$ such that $\varphi_\lambda(x, t)\in Gx$. By assumption, $\Gamma$ is compact. Choose an equivariant Poincar\'{e} 
   system $(D_\gamma, D'_\gamma, p_\gamma, t_\gamma)$ for every essential periodic orbit $\gamma$ of $H$ in such a way that in a neighbourhood $\calU_\gamma$ of $H$, all the Poincar\'{e} maps 
   of elements of $\calU_\gamma$ are defined as maps $D'_\gamma\to D_\gamma$. If $\gamma$ is a periodic orbit of $H_\lambda$, we find an $\eps_\gamma>0$ such that the Poincar\'{e} maps of 
   $H_\mu$ for $\abs{\lambda-\mu}\leq3\eps_\gamma$ constitute a Poincar\'{e} homotopy 
   \[
    P:D'_\gamma\times\incc{\lambda-3\eps_\gamma,\lambda+3\eps_\gamma}\to D_\gamma
   \]
   and so do all elements of $\calU_\gamma$.

   \item Choose open invariant neighbourhoods $W_\gamma\subseteq V_\gamma\subseteq U_\gamma$ of the underlying geometric orbit of $\gamma$ such that $\conj{W}_\gamma\subseteq V_\gamma$ and 
   \[
    \conj{U}_\gamma\subseteq\bigcup_{x\in D_\gamma'}\tilde{\varphi}_\mu(\incc{0, \tilde{t}_\mu(x)}),
   \]
   \[
    \tilde{\varphi}_\mu(x, \incc{0, \tilde{t}_\mu(x)})\subseteq U_\gamma
   \]
   for all flows $\tilde{\varphi}$ of elements in $\calU_\gamma$ and all $x\in\conj{V}_\gamma\cap D_\gamma'$, $\mu\in\incc{\lambda-3\eps_\gamma, \lambda+3\eps_\gamma}$.

   \item The sets $W_\gamma\times\inoo{\lambda-\eps_\gamma,\lambda+\eps_\gamma}$ cover $\Gamma$ and we find a finite subcover, corresponding to orbits $\gamma_1,\dots, \gamma_m$ at parameters $\lambda_1,
   \dots, \lambda_m$. For simplicity, let
   \[
    W_j=W_{\gamma_j}\times\inoo{\lambda_j-3\eps_{\gamma_j},\lambda_j+3\eps_{\gamma_j}}
   \]
   and $\eps_j=\eps_{\gamma_j}$, $j=1,\dots, m$. Define $W=W_1\cup\dots\cup W_m$ and $\calU_1=\calU_{\gamma_1}\cap\dots\cap\calU_{\gamma_m}$. Then $K=\conj{\Omega}\times I-W$ is compact and all periodic 
   orbits of $H$ meeting $K$ are inessential. So the same holds in a neighbourhood of $H$ which we can assume to contain $\calU_1$. We conclude that every $G$-homotopy $H'$ in $\calU_1$ satisfies
   \begin{itemize}

      \item[-] $H'\in\calU$.

      \item[-] All periodic orbits of $H'$ meeting $K$ are inessential.

      \item[-] Lemma \ref{lemma:equipoincarehomotopy} is applicable to $H'$, the sets $V_{\gamma_j}, U_{\gamma_j}$ and the respective Poincar\'{e} system $(D_{\gamma_j}, D'_{\gamma_j}, 
      p'_{\gamma_j}, t'_{\gamma_j})$.
   
   \end{itemize}

   \item Assume that we have constructed a $G$-homotopy $H_k$, $0\leq k\leq m-1$, such that $H_k\in\calU$, all periodic orbits meeting $K$ are inessential and $H_k$ is equivariantly 
   non-degenerate in $\conj{W_1\cup\dots\cup W_k}$. We find a neighbourhood $\calW_k$ of $H_k$ such that every element of $\calW_k$ is equivariantly non-degenerate in $\conj{W_1\cup\dots\cup W_k}$. Now 
   we apply Lemma \ref{lemma:equipoincarehomotopy} to $H_k$, the sets $V_{\gamma_{k+1}}, U_{\gamma_{k+1}}$ and the corresponding Poincar\'{e} system. We find a neighbourhood $\calV_{k+1}$ of the induced 
   Poincar\'{e} homotopy and a map $\chi_{k+1}:\calV_{k+1}\to \frakX_G(M\times\setR, \Omega, a, b)$ with the properties stated in the lemma. So we can use the genericity results for equivariantly non-degenerate 
   homotopies of maps. Choose an equivariantly non-degenerate homotopy 
   \[
    \conj{W_{\gamma_{k+1}}\cap D'_{\gamma_{k+1}}}\times\incc{\lambda_{k+1}-\eps_{k+1},\lambda_{k+1}+\eps_{k+1}}\to D_{\gamma_{k+1}}
   \]
   and extend it to a $G$-homotopy
   \[
    Q:(U_{\gamma_{k+1}}\cap D'_{\gamma_{k+1}})\times\incc{\lambda_{k+1}-3\eps_{k+1},\lambda_{k+1}+3\eps_{k+1}}\to D_{\gamma_{k+1}}
   \]
   that is equal to the Poincar\'{e} homotopy of $H_k$ outside of 
   \[
    V_{\gamma_{k+1}}\cap D'_{\gamma_{k+1}}\times\incc{\lambda_{k+1}-2\eps_{k+1},\lambda_{k+1}+2\eps_{k+1}}.
   \]
   By choosing the initial homotopy close to the Poincar\'{e} homotopy, we can achieve that $Q\in\calV_{k+1}$ and $\chi(Q)\in\calW_k\cap\calU_1$. Define $H_{k+1}=\chi(Q)$. By construction, $H_{k+1}$ is 
   equivariantly non-degenerate on $\conj{W_{k+1}}$. Since $H_{k+1}\in\calW_k$, $H_{k+1}$ is equivariantly non-degenerate on $\conj{W_1\cap\dots\cap W_{k+1}}$. Since $H_{k+1}\in\calU_1$, $H_{k+1}\in
   \calU$ and all periodic orbits meeting $K$ are inessential.

   \item Arriving at $H_m$, this homotopy is an element in $\calU$, equivariantly non-degenerate in $\conj{W_1\cup\dots\cup W_m}\supseteq\conj{\Omega}\times I-K$ and all periodic orbits meeting 
   $K$ are inessential. So $H_m$ is equivariantly non-degenerate, proving our claim.

\end{enumerate}
\end{proof}


\section{Construction of the Index} 
The idea for the construction of the index dates back to a paper of Chow and Mallet--Paret \cite{chowmallet}, who gave a bifurcation theoretical definition of the ordinary Fuller index. 
One assigns the fixed point index of an associated Poincar\'{e} map to an isolated periodic orbit. Using genericity theory, it becomes clear that a corrective factor has to be introduced 
to handle period multiplying bifurcations. 
The main properties are then derived using genericity theory. So all we need for the definition is an index for equivariant maps that counts fixed points, distinguished by symmetry. 
Such an index is easily obtained from an equivariant Lefschetz number in the sense of \cite{laitinen}. Indeed, assume that $f:G\times_H\setB(V)\to G\times_HV$ is a $G$-map without fixed 
points on the boundary. Then the set of fixed points of $f$ is contained in the interior of $G\times_H\setB(V)$ and it is easy to see that $f$ is equivariantly homotopic to a map
$\tilde{f}:G\times_H\setB(V)\to G\times_H\setB(V)$ which is equal to $f$ in a neighbourhood of $\fix(f)$. The equivariant fixed point index of $f$ is then defined to be the equivariant 
Lefschetz number of $\tilde{f}$. 

We recall that an equivariant Lefschetz number of a self-map of a $G$-CW complex can be defined as an element of the tom Dieck ring 
\[
 \setU_G=\bigoplus_{(H)}\setZ\cdot(H),
\]
the sum ranging over the isotropy subgroups of $G$, compare \cite{laitinen} or \cite{ynrohc}, where in the last reference we have to modify the augmentation map in order to count fixed 
points rather than fixed orbits. The basic property of such a Lefschetz number is the fact that non-vanishing of the $(H)$-component implies existence of an orbit of fixed points of type 
at least $(H)$, with respect to the partial order of isotropy types.

Next, we assume that $M$ is a $G$-manifold and $f:M\to M$ a $G$-map with finitely many orbits of fixed points. Then a fixed point index of $f$ is defined via additivity. More precisely, 
take a fixed point $x$ of $f$ and let $\gamma=Gx$. Since $\gamma$ is isolated, there is a pair of tubular neighbourhoods $U_\gamma'\subseteq U_\gamma$ such that $\gamma$ is the unique 
orbit of fixed points of $f$ in $U_\gamma$ and $f(\conj{U_\gamma}')\subseteq U_\gamma$. The local fixed point index of $f$ is the element of $\setU_G$ defined above and is denoted by
\[
 i_G(f, U_\gamma', U_\gamma)\in\setU_G.
\]
Then 
\[
 i_G(f, M)=\sum_\gamma i_G(f, U_\gamma', U_\gamma)
\]     
is the equivariant fixed point index of $f$, the sum ranging over the finitely many orbits of fixed points of $f$. It is also evident, using $G$-homotopy invariance of the Lefschetz number, 
that for general $G$-maps, the fixed point index can be retrieved via approximation. In a similar fashion we can define an equivariant fixed point index $i_G(f, U, M)$ of a $G$-map 
$f:\conj{U}\to M$, where $U$ is an open invariant subset of $M$ such that $f$ has no fixed points on the boundary of $U$.

We turn to the construction of the equivariant Fuller index. The general setup is as follows. Let $M$ be a compact $G$-manifold, $\Omega\subseteq M$ an invariant open subset, $0<a<b<\infty$ real numbers.
Let $\xi$ be a smooth vector field in $\frakX_G(M, \Omega, a, b)$. Assume that $\xi$ has finitely many isolated essential periodic orbits $\gamma_1,\dots,\gamma_n$. Let $p_1,\dots, p_n$ be 
their minimal periods, $T_1,\dots, T_n\in\incc{a,b}$ their periods, such that $T_j=k_j\cdot p_j$ for some $k_j\in\setN$ and every $j\in\{1,\dots, n\}$. Choose equivariant Poincar\'{e} systems 
$(D_j, D'_j, P_j, t_j)$ around any point on $\gamma_j$ for each $j$. The equivariant Fuller index is defined to be
\[
 I_F^G(\xi,\Omega)=\sum_{j=1}^nI_G(\gamma_j)=\sum_{j=1}^n\frac1{k_j}\tensor i_G(P_j, D_j', D_j)\in\setQ\tensor\bigoplus_{(H)}\setZ.
\]
Clearly, the equivariant Fuller index does not depend on the choice of Poincar\'{e} system, because any two of these for the same orbit are joined by an equivariant isotopy (namely the flow of $\xi$).
For arbitrary $G$-vector fields $\xi\in\frakX_G(M, \Omega, a, b)$, we define the equivariant Fuller index to be the limit of the Fuller indices of a sequence of equivariantly non-degenerate fields 
converging to $\xi$. We will establish in the following that the index is well defined. This will come out along the lines when we prove the basic properties of the index, namely homotopy invariance, 
the solution property and additivity. We start with the last.

\begin{proposition}\label{prop:equifulleradd} 
The equivariant Fuller index is additive, that is, if $\Omega_1,\Omega_2$ are disjoint, invariant, open subsets of $\Omega$ such that all essential periodic orbits of $\xi\in\frakX_G(M, \Omega, 
a, b)$ are contained in $(\Omega_1\cup\Omega_2)\times\inoo{a,b}$, then
\[
 I_F^G(\xi, \Omega)=I_F^G(\xi, \Omega_1)+I_F^G(\xi, \Omega_2).
\]
\end{proposition}

\begin{proof} 
Immediate from the definitions.
\end{proof}
If $H:M\times I\to M$ is any equivariantly non-degenerate homotopy, we have finitely many bifurcation parameters. Let $\lambda$ be such a bifurcation parameter, then we call the periodic orbits 
of $H_\lambda$ which are the limit of a sequence of periodic orbits $\gamma_n$ of $H_{\lambda_n}$, $\lambda_n\to\lambda$, $\lambda_n\neq\lambda$, the limit periodic orbits of $H_\lambda$. We 
will choose equivariant Poincar\'{e} systems around these orbits and vary them slightly to obtain Poincar\'{e} homotopies.

\begin{proposition}\label{prop:fullerindexhomotopic} 
Let $\xi_0,\xi_1$ be two equivariantly non-degenerate $G$-vector fields that are $G$-homotopic via an equivariantly non-degenerate homotopy. Then their equivariant Fuller indices are equal.
\end{proposition}

\begin{proof} 
Since the fields are equivariantly non-degenerately homotopic, we can distinguish between two cases.
\begin{enumerate}[1.]

   \item There are no bifurcation parameters in the interval $\incc{\lambda_1,\lambda_2}$. In this case, we choose disjoint equivariant Poincar\'{e} systems around the finitely many periodic 
   orbits of $H_{\lambda_1}$. These systems will be Poincar\'{e} systems around the periodic orbits for $\lambda_1+\eps$ for $\eps>0$ small. The Poincar\'{e} maps induce a non-degenerate 
   homotopy of Poincar\'{e} maps, so the local fixed point indices of the fixed points corresponding to periodic orbits do not change. By compactness of $\incc{\lambda_1,\lambda_2}$, the 
   equivariant Fuller index does not change during this part of the homotopy. 
   \item There is precisely one bifurcation parameter $\lambda\in\inoo{\lambda_1,\lambda_2}$. In this case, $H_\lambda$ is degenerate. Let $\gamma$ be a limit periodic orbit of $H_\lambda$ and 
   $p$ be its minimal period. Let $(D, D', P, t)$ be an equivariant Poincar\'{e} system for $\gamma$, considered with minimal period $p$. By choosing $D$ small enough, there is an extension of $P$ to an 
   equivariant Poincar\'{e} homotopy, again denoted by $P$, for $\mu\in\incc{\lambda-\eps,\lambda+\eps}$ and some $\eps>0$, and we can achieve that the only fixed points of $P_\mu$ lying in $D$ are those 
   on branches converging to $\gamma$. Denote the finitely many branches converging to $\gamma$ from the left of $\lambda$ by
   \[
    \nu_1^k, \dots, \nu_{r_k}^k,
   \]
   where $k$ runs through the integers and indicates that the minimal period of $\nu_j^k$ approaches $k\cdot p$ for $j=1,\dots, r_k$ as the branch approaches $\gamma$. Let $P_-=P(-\eps)$, $P_+=P(\eps)$. 
   We choose small equivariant discs $M_1^k,\dots, M_{r_k}^k$, centered at the fixed points of $P_-$ corresponding to the geometric orbits $\nu_1^k(-\eps), \dots, \nu_{r_k}^k(-\eps)$. Furthermore, we 
   choose equivariant subdiscs ${M'}_1^k\subseteq M_1^k, \dots$, such that $P_-$ restricts to a map ${{M'_j}^k}\to M_j^k$, $j=1,\dots, r_k$, for all $k$ involved, and the iterates of $P_-$ do so as well. 
   We need only finitely many iterates of $P_-$, hence this condition can be fulfilled. We have a homotopy $P$ between $P_-$ and $P_+$ which is equivariantly non-degenerate at every stage except for the 
   parameter $\lambda$. By Theorem \ref{cor:equicodimdisc}, we find a homotopy $P'$ arbitrarily close to $P$ that is non-degenerate and has no fixed points on the union of the boundaries of the 
   discs ${M'}_j^k$. In particular, $P'_-$ has all its fixed points inside of the discs ${M'}_j^k$ for the various $j, k$ and $P'_-$ is admissibly homotopic to $P_-$, i.e. their fixed point indices are 
   equal. But then, also $P_-$ and $P_+$ are admissibly homotopic, so we find 
   \[
    i_G(P_-^k, D', D)=i_G(P_+^k, D', D)
   \]
   for all $k$. For simplicity, write $H_{\lambda-\eps}=H_-$, $H_{\lambda+\eps}=H_+$. We claim that the equivariant Fuller indices are given by the sums
   \[
    I^G_F(H_-, \Omega)=\sum_{\ssmall{n\cdot p\in\incc{a,b}}}\frac1n\cdot i_G(P_-^n, D'),
   \]
   \[
    I^G_F(H_+, \Omega)=\sum_{\ssmall{n\cdot p\in\incc{a,b}}}\frac1n\cdot i_G(P_+^n, D'),
   \]
   which would immediately yield equality of the two terms.

   We calculate
   \[
    I^G_F(H_-, \Omega)=\sum_{\ssmall{j\cdot k\cdot p\in\incc{a,b}}}\sum_{s=1}^{r_{jk}}\frac1j i_G(P_-^{jk}, M_s^k).
   \]
   On the other hand, to calculate the fixed point index of $P_-^n$ in $D'$, note that the branches $\nu_s^k$ bifurcate with period $k\cdot p$ from $\gamma$. Such a bifurcation induces the bifurcation 
   of a $k$-fold covering space of $\setS^1$ in the group quotient. That is, we have $k$ fixed points of the $k$-th iterate of the Poincar\'{e} map $P_-$ bifurcating, all of which have the same local 
   index, since their Poincar\'{e} systems are equivariantly isotopic via the flow. Their index is given by $i_G(P_-^k, M_s^k)$. Hence, we have a contribution of $k\cdot i_G(P_-^k, M_s^k)$ of these 
   fixed points to the fixed point index of $P_-^k$ in $D'$. Clearly, if $k$ divides $n$, then $P_-^n$ has these fixed points in $M_s^k$ as well, and these contribute $k\cdot i_G(P_-^n, M_s^k)$ to the 
   index of $P_-^n$. Summing all these indices up, we obtain
   \begin{eqnarray*}
   \frac1n i_G(P_-^n, D')&=&\sum_{k\cdot j=n}\sum_{s=1}^{r_n}\frac kn\cdot i_G(P_-^n, M_s^k)\\
                         &=&\sum_{k\cdot j=n}\sum_{s=1}^{r_{jk}}\frac1j\cdot i_G(P_-^{jk}, M_s^k).
   \end{eqnarray*}
   This finally gives
   \[
   \sum_{\ssmall{n\cdot p\in\incc{a,b}}}\frac1n i_G(P_-^n, D')=\sum_{\ssmall{j\cdot k\cdot p\in\incc{a,b}}}\sum_{s=1}^{r_{jk}}\frac1j\cdot i_G(P_-^{jk}, 
   M_s^k)=I^G_F(H_-, \Omega).
   \]
   The whole calculation did not depend on the fact that we were working with $H_-$ instead of $H_+$, and we get the same calculation on the right hand side, verifying 
   equality of both equivariant Fuller indices.
\end{enumerate}
   
Since the equivariant Fuller index remains unchanged in both cases and we have only finitely many bifurcation parameters, the proposition follows. 
\end{proof}

This was the hard part of the invariance theorem. The rest follows easily by standard methods.

\begin{lemma} 
If two equivariantly non-degenerate vector fields 
\[
 \xi_0, \xi_1\in\frakX_G(M, \Omega, a, b) 
\]
are equivariantly homotopic, then they are already equivariantly non-degenerate\-ly homotopic.
\end{lemma}

\begin{proof} 
Let $\calU_0$ be a neighbourhood of $\xi_0$ such that all elements of $\calU_0$ are equivariantly non-degenerate. We can furthermore achieve that all elements of $\calU_0$ are pairwise homotopic via a 
homotopy not leaving $\calU_0$. Hence, all elements of $\calU_0$ are equivariantly non-degenerately homotopic. We can find a similar neighbourhood $\calU_1$ of $\xi_1$. Now if $H\in\frakX_G(M\times\setR, \Omega,
a, b)$ is a homotopy joining $\xi_0$ and $\xi_1$, we find an equivariantly non-degenerate homotopy $K$ arbitrarily close to $H$. In particular we can find such a $K$ so that $K_0\in\calU_0$, $K_1\in
\calU_1$. Pasting together $K$ and equivariantly non-degenerate homotopies joining $\xi_0$ with $K_0$ and $K_1$ with $\xi_1$, respectively, we obtain an equivariantly non-degenerate homotopy joining 
$\xi_0$ and $\xi_1$.
\end{proof} 

We have finally arrived at the result that the equivariant Fuller index is a $G$-homotopy invariant.

\begin{theorem} 
The Fuller index is invariant under admissible $G$-homotopies.
\end{theorem}

\begin{proof} 
If $\xi_0, \xi_1\in\frakX_G(M,\Omega, a, b)$ are homotopic equivariant vector fields and the element $H\in\frakX_G(M\times\setR, \Omega, a, b)$ is a $G$-homotopy between them, choose an equivariantly 
non-degenerate homotopy $K\in\frakX_G(M\times\setR, \Omega, a, b)$ such that $K_0$ is in a given neighbourhood $\calU_0$ of $\xi_0$, $K_1$ is in a given neighbourhood $\calU_1$ of $\xi_1$. Since the Fuller index 
is locally constant in $\xi$, the theorem follows from Proposition \ref{prop:fullerindexhomotopic}.
\end{proof}
 
Being a homotopy invariant implies that the index is locally constant in the set of equivariantly non-degenerate fields, so it is well-defined for arbitrary $G$-vector fields by approximation.

\begin{corollary} 
The equivariant Fuller index is locally constant and hence well-defined.
\end{corollary}

We summarize all the properties of the equivariant Fuller index established so far. 

\begin{theorem} 
The equivariant Fuller index has the following properties.
\begin{enumerate}[1.]

   \item It is invariant under admissible $G$-homotopies, i.e. if $H\in\frakX_G(M\times\setR, \Omega, a, b)$, then
   \[
    I^G_F(H_t, \Omega)\equiv\mbox{const.}
   \]

   \item It is additive, i.e. if $\Omega_1,\Omega_2$ are disjoint invariant open subsets of $\Omega$ and all essential periodic orbits of $\xi\in\frakX_G(M, \Omega, a, b)$ are contained in 
   $\Omega_1\cup\Omega_2$, then
   \[
    I^G_F(\xi, \Omega)=I^G_F(\xi, \Omega_1)+I^G_F(\xi, \Omega_2).
   \]
   
   \item It is normalized. If $\xi$ has a single periodic orbit in $\Omega\times\inoo{a,b}$ of periodicity $k$ and $(P, D', D, t)$ is an equivariant Poincar\'{e} system for the orbit, then
   \[
    I_F^G(\xi,\Omega)=\frac1k\tensor i_G(P, D', D).
   \]
   
   \item It has the solution property. If the projection $\pi_{(H)}$ to the $(H)$-component of $\setQ\tensor\setU_G$ of $I^G_F(\xi, \Omega)$ is not zero, then $\xi$ has an essential periodic orbit in 
   $\Omega$ of type at least $(H)$.
   
\end{enumerate}
\end{theorem} 

\begin{proof} 
All statements are obvious, the last one being a trivial consequence of the corresponding result for the fixed point index.
\end{proof}

\end{document}